\documentclass[12pt, a4paper]{amsart}

\usepackage{a4wide}
\usepackage[UKenglish]{babel}
\usepackage{latexsym}
\usepackage{amssymb}
\usepackage{mathrsfs}
\usepackage{graphicx}
\usepackage[latin1]{inputenc}
\usepackage[numbers]{natbib}
\usepackage{amsmath}
\usepackage{amsfonts}
\usepackage{amsthm}
\usepackage{url}
\usepackage{mathptmx}
\usepackage{helvet}
\usepackage{stackrel}
\usepackage{txfonts}

\usepackage{eurosym}
\usepackage{hyperref} 
\hypersetup{pdftitle={},pdfauthor={Oliver Pfaffel},colorlinks=true,linkcolor=black,citecolor=black,urlcolor=black,breaklinks=true}
\usepackage[capitalize]{cleveref}

\usepackage{pdfsync}

\date{}

\title{}

\newcommand{\bX}{\mathbf{X}} 

\newcommand{\eps}{\epsilonup}

\newcommand{\1}{\mathbf{1}}

\newcommand{\R}{\mathbb{R}}

\newcommand{\T}{\mathsf{T}}

\newcommand{\sto}[2]{\stackrel[#2]{#1}{\longrightarrow}}
\newcommand{\ston}[1]{\stackrel[n\to\infty]{#1}{\longrightarrow}}

\newcommand{\tn}[1]{\textnormal{#1}}

\newcommand{\norm}[1]{\left\|#1\right\|}

\newcommand{\bpm}{\begin{pmatrix}}
\newcommand{\epm}{\end{pmatrix}}
\newcommand{\bsm}{\begin{smallmatrix}}
\newcommand{\esm}{\end{smallmatrix}}

\newtheorem{theorem}{Theorem}[section]
\newtheorem{corollary}[theorem]{Corollary}

\theoremstyle{definition}

\newtheorem{remark}[theorem]{Remark}

\numberwithin{equation}{section}

\begin{document}

\title[Sample covariance matrices of non-linear processes with infinite variance]{Eigenvalues of sample covariance matrices of non-linear processes with infinite variance}
\author{Richard A. Davis}
\address{Department of Statistics, Columbia University, New York, NY 10027, USA}
\email{rdavis@stat.columbia.edu}
\author[O. Pfaffel]{Oliver Pfaffel}
\address{TUM Institute for Advanced Study \& Department of Mathematics, Technische Universit\"at M\"unchen, Germany}
\email{o.pfaffel@gmx.de}

\date{}
\begin{abstract}
We study the $k$-largest eigenvalues of heavy-tailed sample covariance matrices of the form $\bX\bX^\T$ in an asymptotic framework, where the dimension of the data and the sample size tend to infinity. To this end, we assume that the rows of $\bX$ are given by independent copies of some stationary process with regularly varying marginals with index $\alpha\in(0,2)$ satisfying large deviation and mixing conditions. We apply these general results to stochastic volatility and GARCH processes.
\end{abstract}
\keywords{Random Matrix Theory, heavy-tailed distribution, dependent entries, largest eigenvalue,  sample covariance matrix, stochastic volatility, GARCH}
\maketitle

\section{Introduction}

In the statistical analysis of high-dimensional data one often tries to reduce its dimensionality while preserving as much of the variation in the data as possible. One important example of such an approach is the Principal Component Analysis (PCA). PCA makes a linear transformation of the data to a new set of variables, the principal components, which are ordered such that the first few retain most of the variation. Therefore one obtains a lower dimensional representation of the data by retaining only the first few principal components.

The variances of the first $k$ principal components are given by the $k$-largest eigenvalues of the covariance matrix. Let us collect the samples of our multivariate data in a $p\times n$ matrix $\bX$, where we refer to $p$ as the dimension of the data and to $n$ as the sample size. In practice, the true underlying covariance matrix is not available, thus one usually replaces it with the sample covariance matrix $\frac{1}{n} \bX\bX^\T$. For more details on PCA we refer the reader to one of the many textbooks available on this topic, see \cite{Anderson2003} or \cite{Jolliffe2002}, for example. 

To account for large high-dimensional data sets, we  study the $k$-largest eigenvalues of the sample covariance matrix when both the dimension of the data $p$ as well as the sample size $n$ go to infinity.  
The field of research that investigates the spectral properties of large dimensional random matrices has become known as {Random Matrix Theory (RMT)}.
There exist several survey articles which stress the close relationship between Random Matrix Theory and multivariate statistics, including PCA, see e.g. \cite{ElKaroui2005} and \cite{Johnstone2007}. Some authors have already employed tools from Random Matrix Theory to correct traditional tests or estimators which fail when the dimension of the data cannot be assumed to be negligible compared to the sample size. For example, Bai et al. \cite{Bai2009} gives corrections on some likelihood ratio tests that even fail even for moderate dimension (around 20), and El Karoui \cite{ElKaroui2008} consistently estimates the spectrum of a large dimensional covariance matrix using Random Matrix Theory.

Davis, Pfaffel and Stelzer \cite{Davis2011} study the $k$-largest eigenvalues of a sample covariance matrix based on observations that come from a high-dimensional linear process with heavy-tailed marginals. It is often necessary to use non-linear instead of linear models to capture the complex dependence structure of the data.
This is particularly true in finance where the log-returns exhibit both non-linearity and heavy-tails. The objective of this paper is to extend some of the results in \cite{Davis2011} to a non-linear setting.

 
In \cref{sec general result} we study non-linear processes with regularly varying tail probabilities with index less than two which satisfy certain large deviation and mixing conditions. We then apply our results to two heavily employed models in finance; stochastic volatility models in \cref{sec sv model} and GARCH($p$,$q$) processes in \cref{sec garch}. More background on these processes and other financial time series models may be found in \cite{Andersen2009}, for example. We assume throughout sections \ref{sec general result} -- \ref{sec garch} that the dimension $p\approx n^\beta$ with $\beta>0$ satisfying $\beta<\frac{2-\alpha}{\alpha-1}$ if $1<\alpha<2$. This restriction is rather general, however, if $\alpha$ is close to $2$, it becomes quite strict. Therefore we will present a result for the largest eigenvalue of $\bX\bX^\T$ in \cref{sec sv model 2} that holds for $p\geq n$ independently of the value of $\alpha$ as long as $0<\alpha<2$.

This article makes heavy use of the theory of regular variation and point processes. A good introduction may be found in Resnick \cite{Resnick2008}, for example.

\section{Eigenvalues of heavy-tailed sample covariance matrices of stationary processes}\label{sec general result}

Throughout this section we assume that $(X_t)$ is a strictly stationary sequence of random variables with marginals that are regularly varying with tail index smaller than two. In other words, there exist a normalizing sequence $(a_n)$ and an $\alpha\in(0,2)$ such that, for any $x>0$,
\begin{align}\label{reg var X}
 nP(|X_0|>a_n x) \ston{} x^{-\alpha}.
\end{align}
Moreover we assume that $X_0$ satisfies the tail balancing condition, i.e., that the limit
\begin{align*}
\lim_{x\to\infty}\frac{P(X_0>x)}{P(|X_0|>x)} \quad\tn{exists.}
\end{align*}
Then we construct our $p\times n$ observation matrix $\bX=(X_{it})_{it}$ as follows: for each $1\leq i\leq p$, let $(X_{it})_{1\leq t\leq n}$ be an independent copy of $(X_{t})_{1\leq t\leq n}$. This means that $(X_{it})_t$ and $(X_t)_t$ have the same distribution, and all rows of $\bX$ are independent. We denote by $\lambda_1,\ldots,\lambda_p\geq 0$ the eigenvalues of $\bX\bX^\T$ and study them via their induced point process
\[ \sum_{i=1}^p\eps_{a_{np}^{-2}\lambda_i}(B) = \left|\left\{1\leq i\leq p:\, a_{np}^{-2}\lambda_i\in B\right\}\right|, \quad B\subseteq(0,\infty). \]
Our motivation to study this problem comes from the statistical analysis of high-dimensional data. Therefore we assume in the following that $p=p_n$ is an integer valued sequence such that $p_n\to\infty$ as $n\to\infty$.\\


The theorem below is a combination of results from \cite{Davis2011}, where mostly linear processes are studied. This general theorem will be the basis of all our further results.

\begin{theorem}\label{generaltheorem}
Assume that there exist $b\geq 0$ and $\alpha\in(0,2)$ such that
\begin{align} pP\left(\sum_{t=1}^n X_t^2>a_{np}^2 x\right) \ston{} b x^{-\alpha/2} \quad\tn{ for each }x>0. \label{largedeviation}\end{align}
Suppose that $p=p_n\to\infty$ and $n\to\infty$ such that
\begin{align}\label{beta condition}
\limsup_{n\to\infty}\frac{p_n}{n^\beta}<\infty 
\end{align}
for some $\beta>0$, satisfying $\beta<\frac{2-\alpha}{\alpha-1}$ if $1<\alpha<2$.
Then we have that
\begin{align}\label{ev_result}
\sum_{i=1}^p\eps_{a_{np}^{-2}\lambda_i} \sto{D}{} \sum_{i=1}^\infty\eps_{b^{2/\alpha}\Gamma_i^{-2/\alpha}}
\end{align}
as $n\to\infty$, where $\Gamma_i=E_1+\ldots+E_i$ is the successive sum of independent and identically distributed (iid) exponential random variables $E_k$ with mean one. 
\end{theorem}

\noindent
The convergence in \eqref{ev_result} means that, for any function $f:(0,\infty)\rightarrow(0,\infty)$ with compact support,
\[ E\left(e^{-\sum_{i=1}^p f(a_{np}^{-2}\lambda_i)}\right) \ston{} E\left(e^{-\sum_{i=1}^\infty f(b^{2/\alpha}\Gamma_i^{-2/\alpha})}\right). \]

\begin{proof}
Proposition 3.3 of \cite{Davis2011} shows that 
\begin{align}
a_{np}^{-2} \norm{\bX\bX^\T-D}_2 \ston{P} 0.
\end{align}
This means that $\bX\bX^\T$ can be approximated by its diagonal. Using Weyl's inequality (\cite[Corollary III.2.6]{Bhatia1997}), its eigenvalues are therefore asymptotically equal to its diagonal entries.
By \cite[Proposition 3.21]{Resnick2008}, the large deviation result \eqref{largedeviation} implies that the point process of the diagonal entries of $\bX\bX^\T$ converges to a Poisson point process,
\begin{align}
\sum_{i=1}^p\eps_{a_{np}^{-2}\sum_{t=1}^n X_{it}^2} \ston{D} \sum_{i=1}^\infty\eps_{b^{2/\alpha}\Gamma_i^{-2/\alpha}}.
\end{align}
Along the lines of the proof of \cite[Theorem 1]{Davis2011} it follows that this results carries over to the eigenvalues, since, as mentioned before, they behave like the diagonal entries.
\end{proof}

\begin{remark}
Observe that \eqref{ev_result} immediately implies the joint convergence of the $k$-largest eigenvalues of $\bX\bX^\T$ in distribution. Denote by $\lambda_{(1)}\geq\ldots\geq\lambda_{(p)}\geq 0$ the eigenvalues of $\bX\bX^\T$ in decreasing order. Then we have, for any fixed integer $k$, that
\[ a_{np}^{-2} (\lambda_{(1)},\ldots,\lambda_{(k)}) \ston{D} b^{2/\alpha}(\Gamma_1^{-2/\alpha},\ldots,\Gamma_k^{-2/\alpha}). \]
If $b=0$, then the normalized eigenvalues converge to zero in probability.
\end{remark}

The large deviation condition \eqref{largedeviation} is essentially equivalent to the convergence of the point process of the partial sums of $X_t^2$  to a limiting point process. 
Instead of having a condition on the partial sums, it would be much more convenient in many cases to have a condition on the process itself.
Davis and Hsing \cite{Davis1995} give very general conditions under which the point process convergence of the sequence $X_t^2$ gives a large deviation result for the partial sums of $X_t^2$. This will be stated in our next theorem.

\begin{theorem}\label{generaltheorem2}
Assume that
\begin{align}\label{cluster}
\sum_{i=1}^n\eps_{a_{n}^{-2}X_{i}^2} \ston{D} \sum_{i=1}^\infty\sum_{j=1}^\infty \eps_{P_i Q_{ij}},
\end{align}
where $\sum_{i=1}^\infty\eps_{P_i}$ is a Poisson process on $(0,\infty)$ with intensity measure $\nu$, and $(\sum_{j=1}^\infty\eps_{Q_{ij}})_i$ is a sequence of iid point processes on $[-1,0)\cup(0,1]$ independent of $\sum_{i=1}^\infty\eps_{P_i}$.
Further assume that the sequence $(X_t)$ is strongly mixing.
If $p,n\to\infty$ such that \eqref{beta condition} is satisfied, then we have \eqref{ev_result} with
\begin{align}\label{eq b}
b=\lim_{\delta\to 0}\int_0^\infty P\left(\sum_{i=1}^\infty u Q_{1i} \1_{(\delta,\infty)}(u|Q_{1i}|)>1\right)\nu(du)\in[0,\infty).
\end{align}
\end{theorem}

\begin{proof}
Under the above conditions, \cite[Theorem 4.3]{Davis1995} shows that
\[ \lim_{n\to\infty}\frac{P\left(\sum_{t=1}^n X_t^2>a_{np}^2x\right)}{nP(X_0^2>a_{np}^2x)}=\lim_{\delta\to 0}\int_0^\infty P\left(\sum_{i=1}^\infty u Q_{1i} \1_{(\delta,\infty)}(u|Q_{1i}|)>1\right)\nu(du)=:b.
\]
Since $pnP(X_0^2>a_{np}^2 x)\to x^{-\alpha/2}$, this implies
\[ pP\left(\sum_{t=1}^n X_t^2>a_{np}^2 x\right) \ston{} b x^{-\alpha/2}. \]
An application of \cref{generaltheorem} completes the proof.
\end{proof}

\begin{remark}
The assumption that the sequence $(X_t)$ is strongly mixing can be replaced by the much weaker assumption \cite[(2.1)]{Davis1995}.
\end{remark}

In the remainder of this article we apply the results of this section to stochastic volatility and GARCH processes. For a stochastic volatility process the constant $b$ in \eqref{ev_result} is essentially one, see \eqref{sv like iid}. In the case of a GARCH process the characterization of $b$ is more involved.

\section{Stochastic Volatility Models}\label{sec sv model}


In this section we use the previously introduced techniques to obtain results for the eigenvalues of sample covariance matrices when the observations are given by stochastic volatility models. More precisely, the rows of the observation matrix $\bX$ are given by independent copies of a univariate stochastic volatility process $X_t=\sigma_t Z_t$, and $\lambda_1,\ldots,\lambda_p$ denote the eigenvalues of $\bX\bX^\T$.

\begin{theorem}\label{sv1}
Assume that $(Z_t)$ is an iid sequence with regularly varying tails with index $\alpha\in(0,2)$ and normalizing sequence $(a_n)$. This means we have that
\begin{align}\label{Z reg var}
nP(|Z_0|>a_{n}x) \ston{} x^{-\alpha} \quad\tn{for any }x>0.
\end{align}
Moreover we assume that $Z_0$ satisfies the tail balancing condition
\begin{align}\label{tail balancing}
\lim_{x\to\infty}\frac{P(Z_0>x)}{P(|Z_0|>x)}=q\in[0,1] \quad\tn{exists.}
\end{align}
Let $\sigma_t\geq 0$ be a stationary sequence with $E\sigma_0^{4\alpha}<\infty$ independent of $(Z_t)$. Then $X_t=\sigma_t Z_t$ defines a stochastic volatility process.
Suppose $p=p_n\ston{}\infty$ such that
\[ 0<\liminf_{n\to\infty}\frac{p_n}{n^{\beta_1}} \quad\tn{ and }\quad \limsup_{n\to\infty}\frac{p_n}{n^{\beta_2}}<\infty \]
for some $0<\beta_1<\beta_2$, where $\beta_2<\frac{2-\alpha}{\alpha-1}$ in case $1<\alpha<2$. 
Then we have, as $n\to\infty$, that
\begin{align}\label{ev_result sv}
\sum_{i=1}^p\eps_{a_{np}^{-2}\lambda_i} \sto{D}{} \sum_{i=1}^\infty\eps_{(E\sigma_0^\alpha)^{2/\alpha}\Gamma_i^{-2/\alpha}}\,,
\end{align}
where $\Gamma_i=E_1+\ldots+E_i$ is the successive sum of iid exponential random variables $E_k$ with mean one.
\end{theorem}

\begin{proof}
Theorem 4.2 of \cite{Mikosch2012} applied to $X_t^2=\sigma_t^2 Z_t^2$ gives that
\[ \lim_{n\to\infty}\frac{P\left(\sum_{t=1}^n X_t^2>a_{np}^2x\right)}{nP(X_0^2>a_{np}^2x)}=1.
\]
Since $pnP(X_0^2>a_{np}^2 x)\to x^{-\alpha/2} E\sigma_0^\alpha $, this implies that
\[ pP\left(\sum_{t=1}^n X_t^2>a_{np}^2\right) \ston{} x^{-\alpha/2} E\sigma_0^\alpha . \]
\cref{generaltheorem} concludes.
\end{proof}

\begin{remark}
Note that the factor $(E\sigma_0^\alpha)^{2/\alpha}$ in the limiting point process in \eqref{ev_result sv} appears due to the fact that $a_n$ is the normalizing sequence of $Z_1$ and not $X_1$. 
If $\tilde a_n$ denotes the normalizing sequence of $X_1$, then we have that
\begin{align}\label{sv like iid}
\sum_{i=1}^p\eps_{\tilde a_{np}^{-2}\lambda_i} \sto{D}{} \sum_{i=1}^\infty\eps_{\Gamma_i^{-2/\alpha}}.
\end{align}
Thus we get the same result for a stochastic volatility process as we get for an iid sequence. The intuitive reason for this is that the dependence within the sequence $(X_t)$ is inherited by the light-tailed sequence $(\sigma_t)$. Hence, the extremes of $(X_t)$ are asymptotically independent.
\end{remark}

A combination of \cite[Theorem 4.6]{Mikosch2012} and \cref{generaltheorem} can be used to obtain analogous results for regularly varying Markov chains.

The additional assumption of $0<\liminf_{n\to\infty}\frac{p_n}{n^{\beta_1}}$ in \cref{sv1} is needed to apply Theorem 4.2 of \cite{Mikosch2012}. We can eliminate this assumption at the expense of 
imposing mixing properties on the volatility sequence $(\sigma_t)$. This is the content of the following theorem.

\begin{theorem}
Let $X_t=\sigma_t Z_t$ be a stochastic volatility process such that $|Z_0|$ is regularly varying satisfying \eqref{Z reg var} and \eqref{tail balancing} with tail index $\alpha\in(0,2)$ and normalizing sequence $(a_n)$. Assume that the volatility sequence $(\sigma_t)$ is independent of $(Z_t)$ and either
\begin{itemize}
\item a stationary sequence of $m$-dependent non-negative random variables such that $E\sigma_0^{\alpha+\delta}<\infty$ for some $\delta>0$, or
\item the exponential of a linear process with Gaussian noise, i.e.,
\[ \log\sigma_t = \sum_{k=-\infty}^\infty\psi_k \xi_{t-k}, \]
where $(\xi_t)$ is a sequence of iid mean-zero Gaussian  random variables and $(\psi_k^2)$ is summable.
\end{itemize}
Suppose that $p,n\to\infty$ such that \eqref{beta condition} is satisfied.
Then we have \eqref{ev_result sv}.
\end{theorem}

\begin{proof}
Depending on the choice of the volatility process $(\sigma_t)$, Theorem 3.1. or Theorem 3.3. of Davis and Mikosch \cite{Davis2001} yield that 
\[ \sum_{i=1}^\infty \eps_{a_n^{-1}X_i} \ston{D} \sum_{i=1}^\infty \eps_{\tilde P_i}, \]
where the limiting point process is Poisson with intensity measure 
\[ \tilde\nu(dx)=\alpha E\sigma_0^\alpha [ q x^{-\alpha-1} \1_{(0,\infty)}(x) + (1-q) (-x)^{-\alpha-1} \1_{(-\infty,0)}(x)]dx.
\]
An application of the continuous mapping theorem yields that
\[ \sum_{i=1}^\infty \eps_{a_n^{-2}X_i^2} \ston{D} \sum_{i=1}^\infty \eps_{\tilde P_i^2} = \sum_{i=1}^\infty \eps_{P_i}. \]
The limiting point process $\sum_{i=1}^\infty \eps_{P_i}$ of the squares has intensity measure 
\[ \nu(dx)=\alpha E\sigma_0^\alpha x^{-\alpha-1} \1_{(0,\infty)}(x)dx\]
and representation \eqref{cluster} with $Q_{ij}=1$ if $i=j$ and zero otherwise. For either choice of $(\sigma_t)$, this sequence is strongly mixing. Davis and Mikosch \cite{Davis2001} have shown that $(X_t)$ is strongly mixing if $(\sigma_t)$ is.
Thus we can apply \cref{generaltheorem2}. Due to the simple structure of $Q$, we have that
\[ b=\lim_{\delta\to 0}\int_0^\infty P\left(\sum_{i=1}^\infty u Q_{1i} \1_{(\delta,\infty)}(u|Q_{1i}|)>1\right)\nu(du)
=- \frac{\alpha}{2} E\sigma_0^\alpha \int_1^\infty u^{-\alpha/2-1} du
=E\sigma_0^\alpha.
\]
\end{proof}

\section{GARCH processes}\label{sec garch}

A GARCH$(p,q)$ process $(X_t)$ with $p\geq 1$ or $q\geq 1$ is given by the equations 
\begin{align}\label{eq garch}
X_t =& \sigma_t Z_t, \\
\sigma_t^2 =& \alpha_0 + \sum_{i=1}^p\alpha_i X_{t-i}^2 + \sum_{j=1}^q \beta_j \sigma_{t-j}^2, \nonumber
\end{align}
where $\alpha_i$ and $\beta_j$ are non-negative coefficients such that $\alpha_0>0$, $\alpha_p>0$ if $p\geq 1$ and $\beta_q>0$ if $q\geq 1$, and $(Z_t)$ is an iid sequence. 
We will make use of the results of \cref{sec general result} to obtain the point process convergence of the eigenvalues of $\bX\bX^\T$ when the rows of $\bX$ are given by iid copies of a GARCH process.

\begin{theorem}
Suppose that $(Z_t)$ is an iid sequence satisfying the following conditions:
\begin{itemize}
\item $Z_1$ is symmetric, i.e., $Z_1\stackrel{D}{=}-Z_1$, and has a strictly positive density on $\R$.
\item There exists an $h_0\in(0,\infty]$ such that $E|Z_1|^h<\infty$ for $h<h_0$ and $E|Z_1|^{h_0}=\infty$.
\end{itemize}
Assume there exists a unique strictly stationary solution $(X_t)$ to equation \eqref{eq garch}.
{Then $X_1$ is regularly varying with index $\alpha>0$ and some normalizing sequence $(a_n)$}. Suppose that $\alpha<2$.\\ 
Then we have, as $p,n\to\infty$ such that \eqref{beta condition} is satisfied, that the convergence in  \eqref{ev_result} holds with $b$ as given in \eqref{eq b}.\\
The intensity measure $\nu$ in \eqref{eq b} is given by $\nu(x,\infty)=\gamma x^{-\alpha}$ for some $\gamma\in(0,1]$, and $(Q_{ij})$ satisfies $\sup_{j}|Q_{ij}|=1$ for all $i\geq 1$. The distribution of $(Q_{ij})$ is described in detail in Theorem 2.8 of Davis and Mikosch \cite{Davis1998}.
\end{theorem}

\begin{proof}
By Basrak, Davis and Mikosch \cite[Corollary 3.5]{Basrak2002} we obtain that $(X_t)$ is strongly mixing and has regularly varying tail probabilities with index $\alpha>0$. Since we assume that $\alpha<2$, we can apply \cref{generaltheorem2}. In \cite[Theorem 2.10]{Basrak2002} it is shown, using the results from \cite{Davis1998}, that the point process $\sum_{i=1}^n\eps_{a_n^{-2}X_i^2}$ has a limiting point process with representation \eqref{cluster}. For details, see also the proof of \cite[Theorem 3.6]{Basrak2002}.
\end{proof}

\begin{remark}
\begin{enumerate}
\item Note that $\gamma\in(0,1]$ is the extremal index of the sequence $(|X_t|)$. For a definition of the extremal index, see for instance \cite{Andersen2009}.
\item In the case when $(X_t)$ is a GARCH$(1,1)$ process, the theorem simplifies considerably. If $\alpha_0,\alpha_1,\beta_1>0$, then $(X_t)$ has a unique strictly stationary solution if and only if
\[ -\infty < E\log(\alpha_1 Z_1^2+\beta_1) < 0, \]
cf.\! \cite[p.47]{Andersen2009}. Let us further assume that $EZ_1=0$ and $EZ_1^2=1$. In this case, the index $\alpha$ of regular variation of $X_1$ is given by the unique solution to the equation 
\[ E[(\alpha_1 Z_1^2+\beta_1)^\alpha]=1.\] 
\end{enumerate}
\end{remark}

To finish this section we want to mention that, using the results from \cite{Basrak2002} and \cite{Davis1998}, the same sort of argument could be applied to the class of solutions of stochastic recurrence equations. 
%

\section{The largest eigenvalue in the case when $p\geq n$}\label{sec sv model 2}

In the previous sections we assumed that \eqref{beta condition} is satisfied for some $\beta>0$ with $\beta<\frac{2-\alpha}{\alpha-1}$ in the case $1<\alpha<2$. This implies for $\alpha\in(1.5,2)$ that $p$ grows slower than $n$. For example, in statistical genetics one is often confronted with the case that the dimension of the data $p$ (e.g., the number of genes) is much larger than the sample size $n$ (e.g., the number of subjects). This may also happen in finance when one considers the problem of optimizing a large portfolio. 
In the next results we consider the case when $p\geq n$. In particular, if there is no extremal clustering, then 
 the largest eigenvalue $\lambda_{\max}$ of $\bX\bX^\T$
is asymptotically equal to the maximum of the squares of the entries of $\bX$.

\begin{theorem}
Let $(X_t)$ be a strictly stationary stochastic process which satisfies \eqref{reg var X} with normalizing sequence $(a_n)$ and tail index $\alpha\in(0,2)$. Assume that $(|X_t|)$ has extremal index equal to one.
Further let $p=p_n=n^\kappa$ for some $\kappa\geq 1$. 
Suppose, for any $x>0$,  that
\[ \lim_{n\to\infty}\frac{P(\sum_{t=1}^n|X_t|>a_{np}x)}{nP(|X_0|>a_{np}x)}=1, \]
and
\[ \lim_{n\to\infty}\frac{P(\sum_{t=1}^n X_t^2>a_{np}^2 x)}{nP(X_0^2>a_{np}^2 x)}=1. \]
Assume the rows of $\bX$ are given by independent copies of $(X_t)$, and denote by $\lambda_{\max}$ the largest eigenvalue of $\bX\bX^\T$. Then we have
\begin{align}\label{p larger than n}
\frac{\lambda_{\max}}{\max_{1\leq i\leq p,1\leq t\leq n} X_{it}^2} \ston{P} 1.
\end{align}
In particular, this implies that
\[ \lim_{n\to\infty} P(\lambda_{\max}\leq a_{np}^2 x) = \exp(-x^{-\alpha/2}). \]
\end{theorem}

\begin{proof}
Recall that, since $(|X_t|)$ has extremal index equal one, we have that 
\[ P\left(\max_{1\leq t\leq n}|X_t|\leq u_n\right)\sim P\left(|X_1|\leq u_n\right)^n, \quad \textnormal{as } n\to\infty,\]
 for any deterministic sequence $u_n\to\infty$.
Thus we obtain with $\max_{1\leq t\leq n}|X_t|\leq\sum_{t=1}^n|X_t|$ that, as $n\to\infty$,
\begin{align*}
1-o(1)=& \frac{P(\max_{1\leq t\leq n}|X_t|>a_{np} \max\{x,y\})}{nP(|X_0|>a_{np}\max\{x,y\})} \\
\leq& \frac{P(\sum_{t=1}^n|X_t|>a_{np}x,\max_{1\leq t\leq n}|X_t|>a_{np}y)}{nP(|X_0|>a_{np}\max\{x,y\})}\leq \frac{P(\sum_{t=1}^n|X_t|>a_{np}\max\{x,y\})}{nP(|X_0|>a_{np}\max\{x,y\})}\sim 1.
\end{align*}
The latter converges to one by assumption, thus
\[ \frac{pP(\sum_{t=1}^n|X_t|>a_{np}x,\max_{1\leq t\leq n}|X_t|>a_{np}y)}{npP(|X_0|>a_{np}\max\{x,y\})}\to 1. \]
Hence
\[ \sum_{i=1}^p \eps_{a_{np}^{-1}(\sum_{t=1}^n|X_{it}|,\max_{1\leq t\leq n}|X_{it}|)} \ston{D} \sum_{i=1}^\infty \eps_{\Gamma_i^{-1/\alpha}(1,1)}. \]
This yields
\begin{align}\label{needed1}
\frac{\norm{X}_\infty}{\max_{i,t}|X_{it}|} = \frac{\max_i \sum_{t=1}^n|X_{it}|}{\max_{i,t}|X_{it}|}\ston{P} 1,
\end{align}
where $\norm{X}_\infty=\max_i\sum_t |X_{it}|$. Likewise, we obtain for the squares
\begin{align}\label{needed2}
\frac{\max_i \sum_{t=1}^n X_{it}^2}{\max_{i,t} X_{it}^2}\ston{P} 1.
\end{align}
Now we proceed like in the proof of \cite[Theorem 2]{Davis2011}. Since $\max_i \sum_{t=1}^n X_{it}^2\leq\lambda_{\max}\leq\norm{X}_\infty^2$, we obtain by \eqref{needed1} and \eqref{needed2} that
\[ \frac{\lambda_{\max}}{\max_{1\leq i\leq p}\sum_{t=1}^n X_{it}^2} \ston{P} 1. \]
\end{proof}

The above theorem applies to stochastic volatility models, as we show in the upcoming corollary. It should not be mistaken as a special case of \cref{sv1} since the growth condition on $p_n$ is different.

\begin{corollary}
Let $X_t=\sigma_t Z_t$ such that $(Z_t)$ is an iid sequence with regularly varying tails satisfying \eqref{Z reg var} and \eqref{tail balancing} with index $0<\alpha<2$ and normalizing sequence $(a_n)$.
Assume that $\sigma_t\geq 0$ is a stationary sequence independent of $(Z_t)$ that is strongly mixing with rate function $r_j=O(j^{-a})$ with $a>1$. Further assume that $E\sigma_0^{4\alpha+\delta}<\infty$ for some $\delta>0$. 
Let $n\to\infty$ and suppose that $p=p_n=n^\kappa$ for some $\kappa\geq 1$.
Then we have \eqref{p larger than n}. This implies that
\[ \lim_{n\to\infty} P\left(\frac{\lambda_{\max}}{a_{np}^2(E\sigma_0^\alpha)^{2/\alpha}}\leq x\right) = \exp(-x^{-\alpha/2}). \]
\end{corollary}

\begin{proof}
Since the square of a stochastic volatility process is again a stochastic volatility process, it suffices to show that
\[ \lim_{n\to\infty}\frac{P(\sum_{t=1}^n|X_t|>a_{np}x)}{nP(|X_0|>a_{np}x)}=1. \]
For $0<\alpha<1$ this was established in \cite[Theorem 4.2]{Mikosch2012} (note that $a_{np}=n^{\epsilon+1/\alpha}$ since $p=n^\kappa$). For $1\leq\alpha<2$ this has been shown for the centered sum. However, since $\kappa\geq 1$ and $\alpha<2$, we have that $nE|X_0|/a_{np}$ converges to zero, therefore a centering is not needed.
\end{proof}

%
%

\section*{Acknowledgements}
R.D. was supported by ARO MURI grant W911NF-12-1-0385. O.P. thanks the Technische Universität München - Institute for Advanced Study, funded by the German Excellence Initiative, and the International Graduate School of Science and Engineering for financial support. 

\bibliographystyle{abbrvnat}

\end{document}